\documentclass[12pt]{article}
\usepackage{latexsym,amssymb,amsmath,amsthm,enumerate,geometry,cite,float,verbatim}
\newtheorem{theorem}{Theorem}

\newtheorem{lemma}[theorem]{Lemma}

\usepackage{lineno}
\usepackage{setspace}
\allowdisplaybreaks

\begin{document}
\onehalfspace

\title{A bound on the dissociation number}
\author{Felix Bock\and Johannes Pardey\and Lucia D. Penso\and Dieter Rautenbach}
\date{}

\maketitle
\vspace{-10mm}
\begin{center}
{\small 
Institute of Optimization and Operations Research, Ulm University, Ulm, Germany\\
\texttt{$\{$felix.bock,johannes.pardey,lucia.penso,dieter.rautenbach$\}$@uni-ulm.de}
}
\end{center}

\begin{abstract}
The dissociation number ${\rm diss}(G)$ of a graph $G$
is the maximum order of a set of vertices of $G$
inducing a subgraph that is of maximum degree at most $1$. 
Computing the dissociation number of a given graph
is algorithmically hard even when restricted to subcubic bipartite graphs.
For a graph $G$ with 
$n$ vertices, 
$m$ edges,
$k$ components, and
$c_1$ induced cycles of length $1$ modulo $3$, we show
${\rm diss}(G)\geq n-\frac{1}{3}\Big(m+k+c_1\Big)$.
Furthermore, we characterize the extremal graphs
in which every two cycles are vertex-disjoint.\\[3mm]
\end{abstract}

\section{Introduction}

We consider finite, simple, and undirected graphs, and use standard terminology.
A set $D$ of vertices of a graph $G$
is a {\it dissociation set} in $G$ if the subgraph $G[D]$ of $G$ induced by $D$
has maximum degree at most $1$.
The {\it dissociation number ${\rm diss}(G)$} of $G$ 
is the maximum order of a dissociation set in $G$.
The dissociation number is algorithmically hard even when restricted, 
for instance, to subcubic bipartite graphs \cite{bocalo,ordofigowe,ya}.
Fast exact algorithms \cite{kakasc}, 
(randomized) approximation algorithms \cite{kakasc,hobu}, and 
fixed parameter tractability \cite{ts}
have been studied for this parameter or its dual, 
the {\it $3$-path (vertex) cover} number.
Several lower bounds on the dissociation number were proposed:
If $G$ is a graph of order $n$ and size $m$, then
\begin{eqnarray}\label{e0}
{\rm diss}(G) & \geq & 
\begin{cases}
\frac{n}{\left\lceil\frac{\Delta+1}{2}\right\rceil} & \mbox{, if $G$ has maximum degree $\Delta$ \cite{brkakase},}\\
\frac{4}{3}\sum\limits_{u\in V(G)}\frac{1}{d_G(u)+1} & \mbox{, if $G$ has no isolated vertex \cite{brkakase},}\\
\sum\limits_{u\in V(G)}\frac{1}{d_G(u)+1}+\sum\limits_{uv\in E(G)}{|N_G[u]\cup N_G[v]|\choose 2}^{-1} & \mbox{, \cite{goharasc},}\\
\frac{n}{2} & \mbox{, if $G$ is outerplanar \cite{brkakase}}\\
\frac{2n}{3} & \mbox{, if $G$ is a tree \cite{brkakase},}\\
\frac{2n}{k+2}-\frac{m}{(k+1)(k+2)}& \mbox{, if $k=\left\lceil\frac{m}{n}\right\rceil-1$ \cite{brjakaseta}, and}\\
\frac{2n}{3}-\frac{m}{6}& \mbox{, \cite{brkakase}.}
\end{cases}
\end{eqnarray}
The results in the present papers were inspired 
by bounds in (\ref{e0}).

Our main result is the following.
\begin{theorem}\label{theorem1}
If $G$ is a graph with 
$n$ vertices, 
$m$ edges,
$k$ components, and
$c_1$ induced cycles of length $1$ modulo $3$, then 
\begin{eqnarray}\label{e1}
{\rm diss}(G) & \geq & n-\frac{1}{3}\Big(m+k+c_1\Big).
\end{eqnarray}
\end{theorem}
Theorem \ref{theorem1} 
generalizes the lower bound $2n/3$ for trees of order $n$ in (\ref{e0}),
strengthens the general lower bound $\frac{2n}{3}-\frac{m}{6}$ 
in (\ref{e0}) for many graphs, and 
almost implies the lower bound $n/2$ for subcubic graphs of order $n$, which follows from the first bound in (\ref{e0}).
In the proof of Theorem \ref{theorem1}, 
graphs in which all cycles are pairwise vertex-disjoint play an essential role.
We call such graphs {\it cycle-disjoint}; 
their components are restricted cactus graphs,
where a {\it cactus} is a connected graph 
in which every block is either a $K_2$ or a cycle.
As a step towards the understanding of all extremal graphs for Theorem \ref{theorem1}, 
we consider the extremal cycle-disjoint graphs in more detail.
We propose three extension operations $(O_1)$, $(O_2)$, and $(O_3)$ applicable to a given graph $G'$, 
attaching a $P_3$ or a cycle of length not $0$ modulo $3$ by a bridge to $G'$, illustrated in Figure \ref{fig1}.
It is easy to see that applying one of these operations to a graph that satisfies (\ref{e1}) with equality
yields a graph that satisfies (\ref{e1}) with equality.
Since $P_3$ and the cycles of lengths not $0$ modulo $3$ satisfy (\ref{e1}) with equality,
this already allows to construct quite a rich family of extremal graphs, yet not all of them.

\begin{figure}[H]
\begin{center}
$\mbox{}$\hfill
\unitlength 1mm 
\linethickness{0.4pt}
\ifx\plotpoint\undefined\newsavebox{\plotpoint}\fi 
\begin{picture}(28,11)(0,0)
\put(25,8){\oval(10,8)[]}
\put(25,8){\makebox(0,0)[cc]{$G'$}}
\put(14,8){\circle*{1}}
\put(14,8){\line(1,0){7}}
\put(17,-5){\makebox(0,0)[cc]{$(O_1)$}}
\put(7,8){\circle*{1}}
\put(0,8){\circle*{1}}
\put(0,8){\line(1,0){14}}
\put(0,11){\makebox(0,0)[cc]{$u$}}
\put(7,11){\makebox(0,0)[cc]{$v$}}
\put(14,11){\makebox(0,0)[cc]{$w$}}
\end{picture}\hfill
\unitlength 1mm 
\linethickness{0.4pt}
\ifx\plotpoint\undefined\newsavebox{\plotpoint}\fi 
\begin{picture}(25,13)(0,0)
\put(22,8){\oval(10,8)[]}
\put(22,8){\makebox(0,0)[cc]{$G'$}}
\put(11,8){\circle*{1}}
\put(11,8){\line(1,0){7}}
\put(14,-5){\makebox(0,0)[cc]{$(O_2)$}}
\put(4,4){\circle*{1}}
\put(4,12){\circle*{1}}
\multiput(4,12)(.058823529,-.033613445){119}{\line(1,0){.058823529}}
\multiput(11,8)(-.058823529,-.033613445){119}{\line(-1,0){.058823529}}
\put(11,11){\makebox(0,0)[cc]{$v$}}
\put(0,12){\makebox(0,0)[cc]{$u$}}
\put(0,4){\makebox(0,0)[cc]{$u'$}}
\end{picture}\hfill
\unitlength 1mm 
\linethickness{0.4pt}
\ifx\plotpoint\undefined\newsavebox{\plotpoint}\fi 
\begin{picture}(28,15)(0,0)
\put(25,8){\oval(10,8)[]}
\put(25,8){\makebox(0,0)[cc]{$G'$}}
\put(7,8){\circle{14}}
\put(14,8){\circle*{1}}
\put(14,8){\line(1,0){7}}
\put(7,8){\makebox(0,0)[cc]{$C_\ell$}}
\put(17,-5){\makebox(0,0)[cc]{$(O_3)$ {\footnotesize (with $\ell\not\equiv 0\,{\rm mod}\, 3$)}}}
\end{picture}
\hfill$\mbox{}$
\end{center}
\caption{Operations constructing an extremal graph from a smaller extremal graph $G'$.}
\label{fig1}
\end{figure}
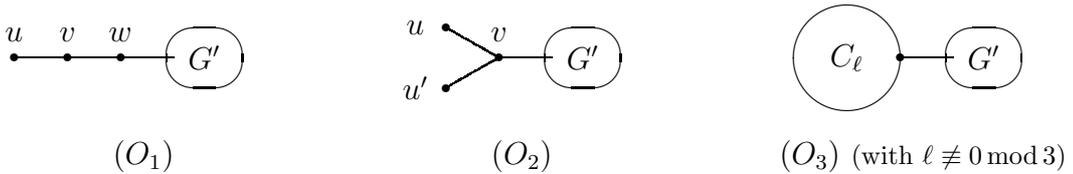
The two operations $(O_1)$ and $(O_2)$ 
are sufficient for the constructive characterization of all trees $T$ 
of order $n$ with ${\rm diss}(T)=2n/3$, that is, 
of all trees that are extremal for the bound from \cite{brkakase} stated in (\ref{e0}).
Let ${\cal T}$ be the set of all trees that arise from $P_3$
by repeated applications of the two operations $(O_1)$ and $(O_2)$,
attaching a new $P_3$ by a bridge to trees in ${\cal T}$.

\begin{theorem}\label{theorem2}
For a tree $T$ of order $n$, the following statements are equivalent.
\begin{enumerate}[(a)]
\item ${\rm diss}(T)=\frac{2n}{3}$.
\item $T\in {\cal T}$.
\item $n\equiv 0\,{\rm mod}\, 3$, and, for every vertex $y$ of $T$,
at most two components of $T-y$ have order not $0$ modulo $3$.
\end{enumerate}
\end{theorem}
Next to the three simple operations illustrated in Figure \ref{fig1},
we introduce one slightly more complicated operation
involving so-called {\it ((very) good) spiked cycles}:
For positive integers $\ell$ and $k$ with $\ell\geq \max\{ 3,k\}$, 
and indices $i_1,\ldots,i_k\in [\ell]$ with $i_1<i_2<\ldots<i_k$,
a {\it spiked cycle $C^*$ with $k$ spikes at $\{ i_1,\ldots,i_k\}$}
arises from the cycle $C:u_1u_2\ldots u_\ell u_1$ of length $\ell$
by attaching a new endvertex $v_{i_j}$ to $u_{i_j}$ for every $j\in [k]$.
The spiked cycle $C^*$ is {\it good} if 
either $k=1$ and $\ell\equiv 1\,{\rm mod}\,3$
or $k\geq 2$,
\begin{itemize}
\item $i_{j+1}-i_j\equiv 2\,{\rm mod}\,3$ for every $j\in [k-1]$, and
\item $\ell+i_1-i_k\equiv 1\,{\rm mod}\,3$,
\end{itemize}
that is, the $k$ paths in $C^*$ between vertices of degree $3$ 
whose internal vertices have degree $2$, 
have lengths $2,\ldots,2$, and $1$ modulo $3$.
The spiked cycle $C^*$ is {\it very good} if it is good and 
\begin{itemize}
\item $\ell\not\equiv 1\,{\rm mod}\,3$,
\end{itemize}
that is, in particular, $k\geq 2$.
See Figure \ref{fig2} for an illustration.

\begin{figure}[H]
\begin{center}
\unitlength 1mm 
\linethickness{0.4pt}
\ifx\plotpoint\undefined\newsavebox{\plotpoint}\fi 
\begin{picture}(59,38)(0,0)
\put(30,18){\oval(50,20)[]}
\put(53,26){\circle*{1}}
\put(53,10){\circle*{1}}
\put(55,18){\circle*{1}}
\put(5,18){\circle*{1}}
\put(7,26){\circle*{1}}
\put(7,10){\circle*{1}}
\put(30,8){\circle*{1}}
\put(35,28){\circle*{1}}
\put(45,28){\circle*{1}}
\put(25,28){\circle*{1}}
\put(15,28){\circle*{1}}
\put(7,26){\line(-1,1){5}}
\put(25,28){\line(0,1){7}}
\put(53,10){\line(1,-1){5}}
\put(38,8){\circle*{1}}
\put(46,8){\circle*{1}}
\put(22,8){\circle*{1}}
\put(14,8){\circle*{1}}
\put(38,8){\line(0,-1){7}}
\put(22,8){\line(0,-1){7}}
\put(2,31){\circle*{1}}
\put(25,35){\circle*{1}}
\put(58,5){\circle*{1}}
\put(38,1){\circle*{1}}
\put(22,1){\circle*{1}}
\put(1,18){\makebox(0,0)[cc]{$u_1$}}
\put(4,7){\makebox(0,0)[cc]{$u_{15}$}}
\put(10,23){\makebox(0,0)[cc]{$u_2$}}
\put(25,38){\makebox(0,0)[cc]{$v_4$}}
\put(25,24){\makebox(0,0)[cc]{$u_4$}}
\put(22,11){\makebox(0,0)[cc]{$u_{13}$}}
\end{picture}
\end{center}
\caption{A very good spiked cycle with $\ell=15$ and $k=5$ 
spikes at $\{ i_1,\ldots,i_k\}=\{ 2,4,9,11,13\}$.
Note that removing $v_4$ results in a spiked cycle that is not good.}
\label{fig2}
\end{figure}
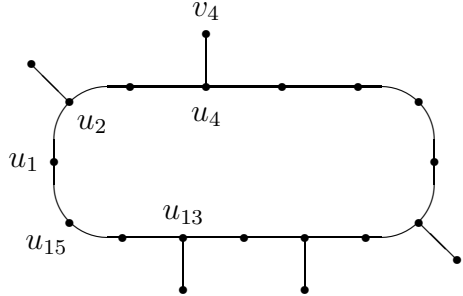
Let ${\cal C}$ be the set of all graphs that arise from the graphs in 
$${\cal C}_0=\Big\{ P_3\Big\}
\cup \Big\{C_\ell: \ell\in \mathbb{N},\, \ell\geq 3,\mbox{ and }\ell\not\equiv 0\,{\rm mod}\, 3\Big\}
\cup \Big\{C^*: C^*\mbox{ is a very good spiked cycle}\Big\}$$
by repeated applications of the three operation $(O_1)$, $(O_2)$, and $(O_3)$,
as well as the fourth operation $(O_4)$ of forming 
the disjoint union of some graph $G'$ in ${\cal C}$ with 
a very good spiked cycle $C^*$, and 
adding a bridge between $V(G')$ and $V(C^*)$.

\begin{lemma}\label{lemma2}
All graphs in ${\cal C}$ satisfy (\ref{e1}) with equality.
Furthermore, for every vertex $u$ of every graph $G$ in ${\cal C}$,
the graph $G$ has a maximum dissociation set not containing $u$.
\end{lemma}
As our final result, we show that ${\cal C}$ 
contains all connected cycle-disjoint extremal graphs for Theorem \ref{theorem1}.
Figure \ref{figex} shows two extremal graphs that are not cycle-disjoint.

\begin{theorem}\label{theorem3}
A connected cycle-disjoint graph 
satisfies (\ref{e1}) with equality if and only if it belongs to ${\cal C}$.
\end{theorem}

\begin{figure}[H]
\begin{center}
\unitlength 0.6mm 
\linethickness{0.4pt}
\ifx\plotpoint\undefined\newsavebox{\plotpoint}\fi 
\begin{picture}(36,46)(0,0)
\put(5,5){\circle*{2}}
\put(35,5){\circle*{2}}
\put(5,25){\circle*{2}}
\put(35,25){\circle*{2}}
\put(5,45){\circle*{2}}
\put(35,45){\circle*{2}}
\put(20,35){\circle*{2}}
\put(5,45){\line(3,-2){15}}
\put(20,35){\line(3,2){15}}
\put(35,45){\line(0,-1){20}}
\put(35,25){\line(-3,2){15}}
\put(20,35){\line(-3,-2){15}}
\put(5,25){\line(0,1){20}}
\put(5,25){\line(0,-1){20}}
\put(35,25){\line(0,-1){20}}
\put(20,29){\makebox(0,0)[cc]{$u$}}
\end{picture}\hspace{2cm}
\unitlength 0.6mm 
\linethickness{0.4pt}
\ifx\plotpoint\undefined\newsavebox{\plotpoint}\fi 
\begin{picture}(31,46)(0,0)
\put(15,25){\circle*{2}}
\put(15,45){\circle*{2}}
\put(0,35){\circle*{2}}
\put(30,35){\circle*{2}}
\put(15,45){\line(3,-2){15}}
\put(0,35){\line(3,2){15}}
\put(15,45){\line(0,-1){20}}
\put(15,25){\line(-3,2){15}}
\put(30,35){\line(-3,-2){15}}
\put(15,20){\makebox(0,0)[cc]{$u$}}
\end{picture}
\end{center}
\caption{Two graphs $G$ that satisfy (\ref{e1}) with equality. 
Note that the removal of the vertex $u$, which lies on two cycles, 
yields $d_G(u)-2$ components.}
\label{figex}
\end{figure}
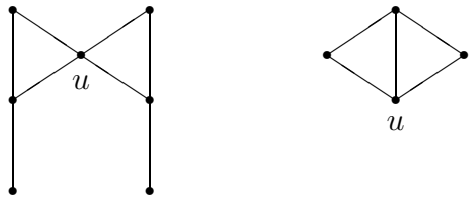
All proofs are given in the next section.

\section{Proofs}

\begin{proof}[Proof of Theorem \ref{theorem1}]
We prove the statement by contradiction, and 
suppose that $G$ is a counterexample of minimum order.
Clearly, this implies that $G$ is connected, that is, we have $k=1$.

If $G$ is not cycle-disjoint, 
then there is a vertex $u$ of $G$ such that 
$G'=G-u$ has $k'\leq d_G(u)-2$ components, and 
the choice of $G$ implies the contradiction
\begin{eqnarray*}
{\rm diss}(G) 
&\geq& {\rm diss}(G')
\geq (n-1)-\frac{1}{3}\Big((m-d_G(u))+k'+c_1(G')\Big)
\geq n-\frac{1}{3}\Big(m+1+c_1\Big),
\end{eqnarray*}
where $c_1(G')$ denotes the number of induced cycles of length $1$ modulo $3$ in $G'$,
and we used the obvious fact that $c_1(G')\leq c_1$.
Hence, the graph $G$ is cycle-disjoint.

Using the bound for trees in (\ref{e0}), and 
${\rm diss}(C_\ell)=\left\lfloor\frac{2\ell}{3}\right\rfloor$ for every integer $\ell\geq 3$,
it follows easily that $G$ is neither a tree nor a cycle.
We consider a longest path $P$, say $P:BvB'\ldots$, 
in the block-cutvertex tree \cite{ba} of $G$,
that is, $B$ and $B'$ are distinct blocks of $G$, 
$v$ is a cutvertex of $G$ that belongs to $B$ and $B'$, and 
all blocks of $G$ that contain $v$ --- except for possibly the block $B'$ --- are endblocks.
Let ${\cal B}$ be the set of all blocks of $G$ 
that contain $v$ and are distinct from $B'$.
Let ${\cal B}$ contain $p$ blocks that are $K_2$s, 
and $q$ blocks that are cycles.
Since $G$ is cycle-disjoint, we have $q\in \{ 0,1\}$.
The graph $G'=G-\bigcup\limits_{H\in {\cal B}}V(H)$ is connected and cycle-disjoint.
Since $B'$ is a $K_2$ or a cycle, 
the number $d$ of neighbors of $v$ in $V(G')$ is $1$ or $2$.
Note that $c_1(G')\leq c_1$, and $c_1(G')\leq c_1-1$ 
if ${\cal B}$ contains a cycle of length $1$ modulo $3$.
See Figure \ref{figbb} for an illustration.

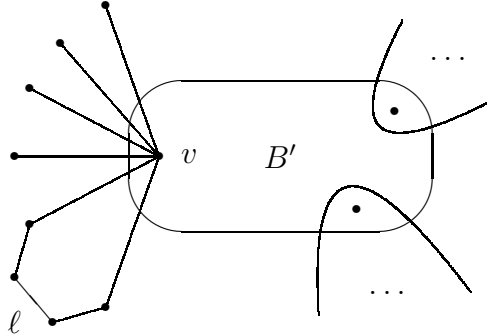
\begin{figure}[H]
\begin{center}
\unitlength 1mm 
\linethickness{0.4pt}
\ifx\plotpoint\undefined\newsavebox{\plotpoint}\fi 
\begin{picture}(68,43.5)(0,0)
\put(41,23){\oval(40,20)[]}
\qbezier(57,41)(45.5,18.5)(68,30)
\qbezier(66,5)(44,34.5)(46,2)
\put(56,29){\circle*{1}}
\put(51,16){\circle*{1}}
\put(63,36){\makebox(0,0)[cc]{$\ldots$}}
\put(55,5){\makebox(0,0)[cc]{$\ldots$}}
\put(25,23){\circle*{1}}
\put(29,23){\makebox(0,0)[cc]{$v$}}
\put(41,23){\makebox(0,0)[cc]{$B'$}}
\put(18,43){\circle*{1}}
\put(18,3){\circle*{1}}
\put(12,38){\circle*{1}}
\put(8,32){\circle*{1}}
\put(8,14){\circle*{1}}
\put(6,23){\circle*{1}}
\multiput(18,43)(.033653846,-.096153846){208}{\line(0,-1){.096153846}}
\multiput(25,23)(-.0336787565,.0388601036){386}{\line(0,1){.0388601036}}
\multiput(8,32)(.063670412,-.0337078652){267}{\line(1,0){.063670412}}
\multiput(8,14)(.063670412,.0337078652){267}{\line(1,0){.063670412}}
\multiput(25,23)(-.033653846,-.096153846){208}{\line(0,-1){.096153846}}
\put(11,1){\circle*{1}}
\put(6,7){\circle*{1}}
\multiput(11,1)(.11666667,.03333333){60}{\line(1,0){.11666667}}
\put(6,1){\makebox(0,0)[cc]{$\ell$}}
\multiput(8,14)(-.03333333,-.11666667){60}{\line(0,-1){.11666667}}
\put(6,7){\line(5,-6){5}}
\put(25,23){\line(-1,0){19}}
\end{picture}
\end{center}
\caption{The local configuration within $G$ where $p=4$ and $q=1$.}\label{figbb}
\end{figure}

First, suppose that $q=1$,
that is, one block in ${\cal B}$ is a cycle $C_\ell$.
Since $G$ is cycle-disjoint, we obtain that $B'$ is a $K_2$, and, hence, $d=1$.
Since $C_\ell$ has a maximum dissociation set avoiding $v$,
the choice of $G$ implies the contradiction
\begin{eqnarray}
{\rm diss}(G) 
&\geq & {\rm diss}(G')+p+\left\lfloor\frac{2\ell}{3}\right\rfloor\label{ex0}\\
&\geq & \left(\underbrace{n-p-\ell}_{n(G')}\right)
-\frac{1}{3}\left(\underbrace{\left(m-d-p-\ell\right)}_{m(G')}+1+c_1(G')\right)
+p+\left\lfloor\frac{2\ell}{3}\right\rfloor\label{ex1}\\
&\geq & \Big(n-\ell\Big)
-\frac{1}{3}\Big(\left(m-1-\ell\right)+1+c_1(G')\Big)
+\left\lfloor\frac{2\ell}{3}\right\rfloor\label{ex2}\\
&\geq & n-\frac{1}{3}\Big(m+1+c_1\Big),\label{ex3}
\end{eqnarray}
where the final inequality uses 
\begin{eqnarray}
-\ell+\frac{\ell+1}{3}-\frac{c_1(G')}{3}+\left\lfloor\frac{2\ell}{3}\right\rfloor
\geq -\frac{c_1}{3},\label{ex4}
\end{eqnarray}
which follows from the relation between $c_1(G')$ and $c_1$ mentioned above.
Hence, no block in ${\cal B}$ is a cycle.

If 
either $p\geq 2$ 
or $p=1$ and $B'$ is a cycle, then $m(G')\leq m-3$,
and, the choice of $G$ implies the contradiction
\begin{eqnarray*}
{\rm diss}(G) 
&\geq & {\rm diss}(G')+p
\geq \Big(n-1-p\Big)
-\frac{1}{3}\Big(m(G')+1+c_1(G')\Big)
+p
\geq n-\frac{1}{3}\Big(m+1+c_1\Big).
\end{eqnarray*}
Hence, we obtain that $p=1$ and that $B'$ is a $K_2$,
which implies that $v$ has degree $2$.
Let $w$ be the unique neighbor of $v$ that is not an endvertex.
The graph $G''=G-N_G[v]=G'-w$ has $k\leq d_G(w)-1$ components, 
and $G''$ has $c_1(G'')\leq c_1$ induced cycles of length $1$ modulo $3$.
Since 
$$m(G'')+k\leq (m-d_G(w)-1)+(d_G(w)-1)=m+1-3,$$ 
the choice of $G$ implies the contradiction
\begin{eqnarray*}
{\rm diss}(G) 
&\geq & {\rm diss}(G'')+2
\geq \Big(n-3\Big)
-\frac{1}{3}\Big(m(G'')+k+c_1(G'')\Big)+2
\geq n-\frac{1}{3}\Big(m+1+c_1\Big),
\end{eqnarray*}
which completes the proof.
\end{proof}
Applied to a subcubic graph,
the first reduction considered in the proof of Theorem \ref{theorem1}
corresponds to the removal of a vertex of degree $3$ that is not a cutvertex.
Repeatedly applying this reduction,
the set of removed vertices is a {\it nonseparating independent set};
a notion that is relevant in the context of feedback vertex sets 
of subcubic graphs \cite{sp,uekago}.

\begin{proof}[Proof of Theorem \ref{theorem2}]
(b) $\Rightarrow$ (a): 
Clearly, $P_3$ satisfies (a).
If $T$ arises from a tree $T'$ that satisfies (a) by applying operation $(O_1)$,
then some maximum dissociation set of $T$ consists of $u$, $v$, and some 
maximum dissociation set of $T'$, which implies that $T$ satisfies (a).
Similarly,
if $T$ arises from a tree $T'$ that satisfies (a) by applying operation $(O_2)$,
then some maximum dissociation set of $T$ consists of $u$, $u'$, and some 
maximum dissociation set of $T'$, which implies that $T$ satisfies (a).
A simple inductive argument implies that all trees in ${\cal T}$ satisfy (a).

\medskip

\noindent (a) $\Rightarrow$ (c): Let $T$ satisfy (a).
By induction on the order $n$ of $T$, we prove (c).
Since $P_3$ is the only star that satisfies (a) and $P_3$ satisfies (c), 
we may assume that $n\geq 4$ and that $T$ has diameter at least $3$.
Let $P:uvwx\ldots$ be a longest path in $T$.
Since 
$$\frac{2n}{3}
={\rm diss}(T)
\geq |N_T(v)\setminus \{ w\}|+{\rm diss}\Big(T-(N_T[v]\setminus \{ w\})\Big)
\stackrel{(\ref{e0})}{\geq} (d_T(v)-1)+\frac{2}{3}(n-d_T(v)),$$
we obtain $d_T(v)\in \{ 2,3\}$.

First, suppose that $d_T(v)=2$.
Let $T_1,\ldots,T_p$ be the components of $T-\{ u,v,w\}$,
and let $n_i$ be the order of $T_i$.
Since 
$$\frac{2n}{3}
={\rm diss}(T)
\geq |\{ u,v\}|+\sum_{i=1}^p{\rm diss}(T_i)
\stackrel{(\ref{e0})}{\geq} 2+\sum_{i=1}^p\frac{2n_i}{3}
=\frac{2n}{3},$$
equality holds throughout this inequality chain,
which implies that each $T_i$ satisfies (a).
By induction, each $T_i$ satisfies (c).
Now, let $y$ be any vertex of $T$.
If $d_T(y)\leq 2$, then $T-y$ has at most two components.
Now, let $d_T(y)\geq 3$.
If $y\in V(T_j)$, then the order of the component of $T-y$
that contains $w$ is 
either $3+\sum_{i\not=j}n_i$ if $y$ is the neighbor of $w$ in $V(T_i)$
or $n(K)+3+\sum_{i\not=j}n_i$, where $K$ is the component of $T_i-y$
that contains the neighbor of $w$ in $V(T_i)$.
Since each $n_i$ is $0$ modulo $3$,
the term $3+\sum_{i\not=j}n_i$ is $0$ modulo $3$.
Since $T_i-y$ has at most two components of order not $0$ modulo $3$,
this implies that also $T-y$ has at most two components of order not $0$ modulo $3$.
Finally, if $y=w$, then the only component of $T-y$
of order not $0$ modulo $3$ consists of $u$ and $v$.
Altogether, we obtain that $T$ satisfies (c).

Next, suppose that $d_T(v)=3$.
Since 
$$\frac{2n}{3}
={\rm diss}(T)
\geq |N_T(v)\setminus \{ w\}|+{\rm diss}\Big(T-(N_T[v]\setminus \{ w\})\Big)
\stackrel{(\ref{e0})}{\geq} 2+\frac{2(n-3)}{3}
=\frac{2n}{3},$$
the tree $T-(N_T[v]\setminus \{ w\})$ satisfies (a), and, 
hence, by induction, also (c).
Arguing similarly as above, it follows easily that $T$ satisfies (c).

\medskip

\noindent (c) $\Rightarrow$ (b): Let $T$ satisfy (c).
By induction on the order $n$ of $T$, we prove (b).
Since $P_3$ is the only star that satisfies (c) and $P_3$ satisfies (b), 
we may assume that $n\geq 4$ and that $T$ has diameter at least $3$.
Let $v$ be a vertex of degree at least $2$ 
such that all but exactly one neighbor $w$ of $v$ are endvertices.
Since $T-v$ has $d_T(v)-1$ components of order $1$, 
we obtain, by (c), that $d_T(v)\in \{ 2,3\}$.
If $d_T(v)=3$, 
then it is easy to see that $T'=T-(N_T[v]\setminus \{ w\})$ satisfies (c), and,
hence, by induction, also (b).
Since $T$ arises from $T'$ by applying operation $(O_2)$,
it follows in this case that $T$ satisfies (b).
By symmetry, we may assume that every vertex $v$ of degree at least $2$,
such that all but exactly one neighbor of $v$ are endvertices, has degree $2$.
Let $P:uvwx\ldots$ be a longest path in $T$, and 
let $T''=T-\{ u,v,w\}$.
Since $T$ satisfies (c), 
it follows easily that each component of $T''$ satisfies (c), and, 
hence, by induction, also (b).
If $T''$ is connected, then $T$ arises from $T''$ by applying operation $(O_1)$,
and it follows that $T$ satisfies (b).
Hence, we may assume that $T''$ has at least two components.
By the choice of $P$, 
this implies that in $T$ all vertices of some component $K$ of $T''$ 
are within distance at most $2$ from $w$.
Since $n(K)\geq 3$,
the neighbor $v'$ of $w$ in $V(K)$ is of degree at least $3$,
and all but exactly one neighbor of $v'$ are endvertices, 
which is a contradiction and completes the proof.
\end{proof}
The trees in ${\cal T}$ have the following useful property.

\begin{lemma}\label{lemma0}
For every vertex $u$ of every tree $T$ of order $n$ with ${\rm diss}(T)=2n/3$,
the tree $T$ has a maximum dissociation set not containing $u$.
\end{lemma}
\begin{proof}
The proof is by induction on $n$.
For $n=3$, the statement is obvious.
Now, let $n>3$.
By Theorem \ref{theorem2},
the tree $T$ arises from the disjoint union of a tree $T'$ of order $n'$ 
with ${\rm diss}(T)=2n'/3$ and a copy of $P_3$
by adding a bridge between some vertex $x$ in $T'$
and some vertex $y$ in the $P_3$.
Now, let $u$ be any vertex of $T$.
If $u$ is a vertex of $T'$, 
then adding the two vertices of the $P_3$ that are distinct from $y$
to a maximum dissociation set of $T'$ not containing $u$
yields a maximum dissociation set of $T$ not containing $u$.
If $u$ is a vertex of the $P_3$,
then adding the two vertices of the $P_3$ that are distinct from $u$
to a maximum dissociation set of $T'$ not containing $x$
yields a maximum dissociation set of $T$ not containing $u$.
\end{proof}

\begin{lemma}\label{lemma1}
If $C^*$ is a spiked cycle with $k$ spikes of order $n$,
then ${\rm diss}(C^*)\geq \frac{2n-1}{3}$
with equality if and only if $C^*$ is good.
Furthermore, if $C^*$ is good and $u$ is a vertex of $C^*$
such that, for $k=1$, the degree of $u$ is at least $2$,
then the good spiked cycle $C^*$ has a maximum dissociation set not containing $u$.
\end{lemma}
\begin{proof}
Since all statements are easily verified for $k=1$, we assume now that $k\geq 2$.
Let $C^*$ be a spiked cycle with $k$ spikes at $\{ i_1,\ldots,i_k\}$,
where we use the notation from the definition of spiked cycles.
The graph $T=C^*-\{u_{i_1},v_{i_1}\}$ is a tree of order $n-2$, and
we obtain that
\begin{eqnarray}\label{e2}
{\rm diss}(C^*)
&\geq& {\rm diss}(T)+|\{ v_{i_1}\}|
\stackrel{(\ref{e0})}{\geq} \frac{2(n-2)}{3}+1=\frac{2n-1}{3}.
\end{eqnarray}
Now, suppose that (\ref{e2}) holds with equality throughout.
This implies that $n\equiv 2\,{\rm mod}\,3$ and 
that ${\rm diss}(T)=\frac{2(n-2)}{3}$.
By Theorem \ref{theorem2}, the tree $T$ satisfies (c) of Theorem \ref{theorem2}.
A path in $C^*$ between vertices of degree $3$ 
whose internal vertices have degree $2$ is called {\it special}.
If $i_2-i_1\equiv 0\,{\rm mod}\,3$,
then $T-u_{i_2}$ has three components of order not $0$ modulo $3$,
which is a contradiction.
See Figure \ref{figi1i2} for an illustration.

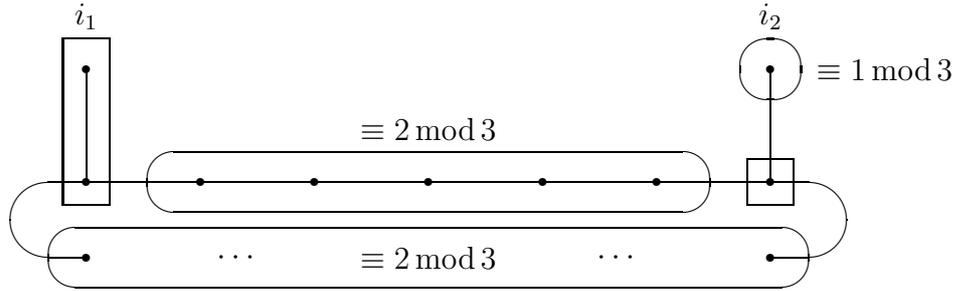
\begin{figure}[H]
\begin{center}
\unitlength 1mm 
\linethickness{0.4pt}
\ifx\plotpoint\undefined\newsavebox{\plotpoint}\fi 
\begin{picture}(110,35)(0,0)
\put(10,13){\circle*{1}}
\put(10,28){\circle*{1}}
\put(25,13){\circle*{1}}
\put(40,13){\circle*{1}}
\put(55,13){\circle*{1}}
\put(70,13){\circle*{1}}
\put(85,13){\circle*{1}}
\put(100,13){\circle*{1}}
\put(100,28){\circle*{1}}
\put(10,13){\line(1,0){90}}
\put(10,28){\line(0,-1){15}}
\put(100,28){\line(0,-1){15}}
\put(9.5,8){\oval(19,10)[lb]}
\put(100.5,8){\oval(19,10)[rt]}
\put(9.5,8){\oval(19,10)[lt]}
\put(100.5,8){\oval(19,10)[rb]}
\put(55,3){\oval(100,8)[]}
\put(100,28){\oval(8,8)[]}
\put(55,13){\oval(74,8)[]}
\put(7,10){\framebox(6,22)[cc]{}}
\put(10,35){\makebox(0,0)[cc]{$i_1$}}
\put(100,35){\makebox(0,0)[cc]{$i_2$}}
\put(55,20){\makebox(0,0)[cc]{$\equiv 2\,{\rm mod}\,3$}}
\put(55,3){\makebox(0,0)[cc]{$\equiv 2\,{\rm mod}\,3$}}
\put(30,3){\makebox(0,0)[cc]{$\cdots$}}
\put(80,3){\makebox(0,0)[cc]{$\cdots$}}
\put(115,28){\makebox(0,0)[cc]{$\equiv 1\,{\rm mod}\,3$}}
\put(10,3){\circle*{1}}
\put(100,3){\circle*{1}}
\put(97,10){\framebox(6,6)[cc]{}}
\end{picture}
\end{center}
\caption{$T-u_{i_1}=C^*-\{ u_{i_1},v_{i_1},u_{i_2}\}$.}\label{figi1i2}
\end{figure}
Hence, by symmetry, no special path has length $0$ modulo $3$.
If $i_2-i_1,i_{j+1}-i_j\equiv 1\,{\rm mod}\,3$, and
$i_3-i_2,i_4-i_3,\ldots,i_j-i_{j-1}\equiv 2\,{\rm mod}\,3$
for some $j\in \{ 2,\ldots,k-1\}$,
then $T-u_{i_{j+1}}$ has three components of order not $0$ modulo $3$,
which is a contradiction.
Hence, by symmetry, at most one special path has length $1$ modulo $3$.
Since $n\equiv 2\,{\rm mod}\,3$, not all special paths have lengths $2$ modulo $3$.
Altogether, we obtain that exactly one special path has length $1$ modulo $3$
while the other $k-1$ special paths have lengths $2$ modulo $3$,
that is, the spiked cycle $C^*$ is good.

Next, suppose that the spiked cycle $C^*$ is good.
By symmetry, we may assume $i_2-i_1\equiv 1\,{\rm mod}\,3$.
This easily implies that some maximum dissociation set $D$ of $C^*$ 
does not contain both $u_{i_1}$ as well as $u_{i_2}$.
By symmetry, we may assume that $D$ does not contain $u_{i_1}$.
Again, let $T=C^*-\{u_{i_1},v_{i_1}\}$.
In view of $D$, we have ${\rm diss}(C^*)={\rm diss}(T)+1$.
Since $C^*$ is good, it is easy to see that $T$ satisfies (c) of Theorem \ref{theorem2}.
Hence, by Theorem \ref{theorem2}, we obtain
${\rm diss}(C^*)={\rm diss}(T)+1=\frac{2(n-2)}{3}+1=\frac{2n-1}{3}$.
By Lemma \ref{lemma0},
the tree $T$ has a maximum dissociation set avoiding any specified vertex.
This easily implies that $C^*$ has a maximum dissociation set 
avoiding any specified vertex distinct from $v_{i_1}$.
Since $\tilde{C}^*=C^*-v_{i_1}$ is a spiked cycle of order $n-1\equiv 1\,{\rm mod}\,3$,
we have ${\rm diss}\left(\tilde{C}^*\right)\geq \left\lceil\frac{2(n-1)-1}{3}\right\rceil=\frac{2n-1}{3}$, 
which implies that a maximum dissociation set of $\tilde{C}^*$
is a maximum dissociation set of $C^*$ avoiding $v_{i_1}$.
This completes the proof.
\end{proof}

\begin{proof}[Proof of Lemma \ref{lemma2}]
By Lemma \ref{lemma1}, the graphs in ${\cal C}_0$ satisfy (\ref{e1}) with equality
and they have maximum dissociation sets avoiding any specified vertex.
Recall that the four operations $(O_1)$ to $(O_4)$
consist in adding a disjoint copy of a graph from ${\cal C}_0$ to some graph $G'$
and connecting this copy by a bridge to $G'$.
It follows that applying one of the four operations to a graph that satisfies (\ref{e1}) with equality
yields a graph that satisfies (\ref{e1}) with equality.
Now, an inductive argument implies 
that all graphs in ${\cal C}$ satisfy (\ref{e1}) with equality.
The existence of maximum dissociation sets avoiding specified vertices
follows easily by induction arguing as in the proof of Lemma \ref{lemma0}
and using Lemma \ref{lemma1}.
\end{proof}

\begin{proof}[Proof of Theorem \ref{theorem3}]
We say that a connected cycle-disjoint graph is {\it extremal} 
if it satisfies (\ref{e1}) with equality.
By Lemma \ref{lemma2}, all graphs in ${\cal C}$ are extremal.
For the converse, let $G$ be extremal.
By induction on the order $n$, we show that $G\in {\cal C}$.
If $G$ is a tree, 
then Theorem \ref{theorem2} implies $G\in {\cal T}\subseteq {\cal C}$.
If $G$ is a cycle of length $\ell$, then $\ell$ is not $0$ modulo $3$
and $G\in {\cal C}_0\subseteq {\cal C}$.
If $G$ is a spiked cycle, then Lemma \ref{lemma1} implies that $G$ is good.
If $G$ is not very good, then $c_1=1$, contradicting the fact that $G$ is extremal. 
Hence, the graph $G$ is a very good spiked cycle, 
and $G\in {\cal C}_0\subseteq {\cal C}$.

Now, let $G$ be neither a tree nor a cycle nor a spiked cycle.
We choose 
$P:BvB'\ldots$,
${\cal B}$,
$p$,
$q\in \{ 0,1\}$, 
$G'$, 
$d$, and
$c_1(G')$ exactly as in the proof of Theorem \ref{theorem1}.

First, we assume that $q=1$.
Since $G$ is cycle-disjoint, 
this implies that $G$ arises by adding the bridge $B'$ 
between $\bigcup\limits_{H\in {\cal B}}H$ and $G'$.
Arguing as in (\ref{ex0}) to (\ref{ex3}) using (\ref{ex4}), 
we obtain that all five inequalities (\ref{ex0}) to (\ref{ex4}) hold with equality.
This implies 
that $p=0$,
that $G'$ is extremal, and 
that $\ell\not\equiv 0\,{\rm mod}\, 3$.
By induction, we obtain that $G'\in {\cal C}$.
Since $G$ is constructed by applying operation $(O_3)$ to $G'$,
we obtain $G\in {\cal C}$.
Hence, we may assume that $q=0$.

Next, we assume that $p\geq 2$.
Note that $p+d\geq 3$.
By Theorem \ref{theorem1}, we obtain 
\begin{eqnarray*}
{\rm diss}(G) 
&\geq & p+{\rm diss}(G')\\
&\geq & p+\Big(n-p-1\Big)-\frac{1}{3}\Big((m-p-d)+1+c_1(G')\Big)\\
&\geq & n-\frac{1}{3}\Big(m+1+c_1\Big).
\end{eqnarray*}
Since equality holds throughout this inequality chain, we obtain 
that $G'$ is extremal, and that $p+d=3$, 
which implies $p=2$ and $d=1$.
By induction, we obtain that $G'\in {\cal C}$.
Since $G$ is constructed by applying operation $(O_2)$ to $G'$,
we obtain $G\in {\cal C}$.
Hence, we may assume that $p=1$.

Next, we assume that $v$ does not lie on a cycle,
that is, the block $B'$ is a $K_2$ and the degree of $v$ is $2$.
Let $w$ be the neighbor of $v$ in $B'$.
Let $G''=G-N_G[v]=G'-w$.
If $G''$ has $k''$ components, then $k''\leq d_G(w)-1$.
By Theorem \ref{theorem1}, we obtain 
\begin{eqnarray*}
{\rm diss}(G) 
&\geq & 2+{\rm diss}(G'')\\
&\geq & 2+\Big(n-3\Big)-\frac{1}{3}\Big((m-d_G(w)-1)+k''+c_1(G'')\Big)\\
&\geq & n-\frac{1}{3}\Big(m+1+c_1\Big).
\end{eqnarray*}
Since equality holds throughout this inequality chain,
each component of $G''$ is extremal, and 
$k''=d_G(w)-1$, 
which implies that $w$ is connected by a bridge to each component of $G''$.
By induction, each component of $G''$ lies in ${\cal C}$.
If $k''=1$, then $G$ is constructed by applying operation $(O_1)$ to $G''$,
and we obtain $G\in {\cal C}$.
Hence, we may assume that $k''=2$.
By symmetry, considering an alternative choice for the path $P$,
we may assume that one component $K$ of $G''$ has order $2$,
which contradicts $K\in {\cal C}$.
Hence, we may assume that $v$ lies on a cycle,
that is, the block $B'$ is a cycle.

Since $G$ is not a spiked cycle, it follows, by symmetry, 
considering alternative choices for the path $P$,
that $G$ arises from the disjoint union of 
\begin{itemize}
\item a spiked cycle $G_0$ of order $n_0$ whose unique cycle is $B'$,
\item a connected cycle-disjoint graph $G_1$ of order $n_1$, and
\item a set $S$ of $s$ isolated vertices,
\end{itemize}
with $n_1+s\geq 2$, 
by adding all possible $1+s$ edges between a vertex $w$ of $G_0$ with $d_{G_0}(w)=2$ and all $1+s$ vertices in $\{ x\}\cup S$, where $x$ is some vertex of $G_1$.
See Figure \ref{figg0g1} for an illustration;
the indicated internal structure of $G_1$ is relevant only later.
Note that $m=m(G_0)+m(G_1)+s+1$.

\begin{figure}[H]
\begin{center}
\unitlength 1.3mm 
\linethickness{0.4pt}
\ifx\plotpoint\undefined\newsavebox{\plotpoint}\fi 
\begin{picture}(78,29)(0,0)
\put(38,19){\line(0,1){.4958}}
\put(37.986,19.496){\line(0,1){.4943}}
\put(37.945,19.99){\line(0,1){.4913}}
\put(37.877,20.481){\line(0,1){.4868}}
\multiput(37.782,20.968)(-.024351,.09616){5}{\line(0,1){.09616}}
\multiput(37.66,21.449)(-.024676,.078894){6}{\line(0,1){.078894}}
\multiput(37.512,21.922)(-.024844,.066356){7}{\line(0,1){.066356}}
\multiput(37.338,22.387)(-.024904,.056775){8}{\line(0,1){.056775}}
\multiput(37.139,22.841)(-.024883,.049171){9}{\line(0,1){.049171}}
\multiput(36.915,23.284)(-.024799,.042953){10}{\line(0,1){.042953}}
\multiput(36.667,23.713)(-.024661,.037747){11}{\line(0,1){.037747}}
\multiput(36.396,24.128)(-.024478,.033303){12}{\line(0,1){.033303}}
\multiput(36.102,24.528)(-.0242542,.0294503){13}{\line(0,1){.0294503}}
\multiput(35.787,24.911)(-.0258398,.0280695){13}{\line(0,1){.0280695}}
\multiput(35.451,25.276)(-.0253935,.0247031){14}{\line(-1,0){.0253935}}
\multiput(35.096,25.622)(-.0287708,.0250565){13}{\line(-1,0){.0287708}}
\multiput(34.722,25.947)(-.032616,.025386){12}{\line(-1,0){.032616}}
\multiput(34.33,26.252)(-.037053,.025692){11}{\line(-1,0){.037053}}
\multiput(33.923,26.534)(-.038412,.023612){11}{\line(-1,0){.038412}}
\multiput(33.5,26.794)(-.04362,.023606){10}{\line(-1,0){.04362}}
\multiput(33.064,27.03)(-.049838,.023519){9}{\line(-1,0){.049838}}
\multiput(32.615,27.242)(-.05744,.02333){8}{\line(-1,0){.05744}}
\multiput(32.156,27.429)(-.067015,.023006){7}{\line(-1,0){.067015}}
\multiput(31.687,27.59)(-.079544,.022493){6}{\line(-1,0){.079544}}
\put(31.209,27.725){\line(-1,0){.484}}
\put(30.725,27.833){\line(-1,0){.4892}}
\put(30.236,27.915){\line(-1,0){.493}}
\put(29.743,27.969){\line(-1,0){.4952}}
\put(29.248,27.997){\line(-1,0){.496}}
\put(28.752,27.997){\line(-1,0){.4952}}
\put(28.257,27.969){\line(-1,0){.493}}
\put(27.764,27.915){\line(-1,0){.4892}}
\put(27.275,27.833){\line(-1,0){.484}}
\multiput(26.791,27.725)(-.079544,-.022493){6}{\line(-1,0){.079544}}
\multiput(26.313,27.59)(-.067015,-.023006){7}{\line(-1,0){.067015}}
\multiput(25.844,27.429)(-.05744,-.02333){8}{\line(-1,0){.05744}}
\multiput(25.385,27.242)(-.049838,-.023519){9}{\line(-1,0){.049838}}
\multiput(24.936,27.03)(-.04362,-.023606){10}{\line(-1,0){.04362}}
\multiput(24.5,26.794)(-.038412,-.023612){11}{\line(-1,0){.038412}}
\multiput(24.077,26.534)(-.037053,-.025692){11}{\line(-1,0){.037053}}
\multiput(23.67,26.252)(-.032616,-.025386){12}{\line(-1,0){.032616}}
\multiput(23.278,25.947)(-.0287708,-.0250565){13}{\line(-1,0){.0287708}}
\multiput(22.904,25.622)(-.0253935,-.0247031){14}{\line(-1,0){.0253935}}
\multiput(22.549,25.276)(-.0258398,-.0280695){13}{\line(0,-1){.0280695}}
\multiput(22.213,24.911)(-.0242542,-.0294503){13}{\line(0,-1){.0294503}}
\multiput(21.898,24.528)(-.024478,-.033303){12}{\line(0,-1){.033303}}
\multiput(21.604,24.128)(-.024661,-.037747){11}{\line(0,-1){.037747}}
\multiput(21.333,23.713)(-.024799,-.042953){10}{\line(0,-1){.042953}}
\multiput(21.085,23.284)(-.024883,-.049171){9}{\line(0,-1){.049171}}
\multiput(20.861,22.841)(-.024904,-.056775){8}{\line(0,-1){.056775}}
\multiput(20.662,22.387)(-.024844,-.066356){7}{\line(0,-1){.066356}}
\multiput(20.488,21.922)(-.024676,-.078894){6}{\line(0,-1){.078894}}
\multiput(20.34,21.449)(-.024351,-.09616){5}{\line(0,-1){.09616}}
\put(20.218,20.968){\line(0,-1){.4868}}
\put(20.123,20.481){\line(0,-1){.4913}}
\put(20.055,19.99){\line(0,-1){.4943}}
\put(20.014,19.496){\line(0,-1){1.4859}}
\put(20.055,18.01){\line(0,-1){.4913}}
\put(20.123,17.519){\line(0,-1){.4868}}
\multiput(20.218,17.032)(.024351,-.09616){5}{\line(0,-1){.09616}}
\multiput(20.34,16.551)(.024676,-.078894){6}{\line(0,-1){.078894}}
\multiput(20.488,16.078)(.024844,-.066356){7}{\line(0,-1){.066356}}
\multiput(20.662,15.613)(.024904,-.056775){8}{\line(0,-1){.056775}}
\multiput(20.861,15.159)(.024883,-.049171){9}{\line(0,-1){.049171}}
\multiput(21.085,14.716)(.024799,-.042953){10}{\line(0,-1){.042953}}
\multiput(21.333,14.287)(.024661,-.037747){11}{\line(0,-1){.037747}}
\multiput(21.604,13.872)(.024478,-.033303){12}{\line(0,-1){.033303}}
\multiput(21.898,13.472)(.0242542,-.0294503){13}{\line(0,-1){.0294503}}
\multiput(22.213,13.089)(.0258398,-.0280695){13}{\line(0,-1){.0280695}}
\multiput(22.549,12.724)(.0253935,-.0247031){14}{\line(1,0){.0253935}}
\multiput(22.904,12.378)(.0287708,-.0250565){13}{\line(1,0){.0287708}}
\multiput(23.278,12.053)(.032616,-.025386){12}{\line(1,0){.032616}}
\multiput(23.67,11.748)(.037053,-.025692){11}{\line(1,0){.037053}}
\multiput(24.077,11.466)(.038412,-.023612){11}{\line(1,0){.038412}}
\multiput(24.5,11.206)(.04362,-.023606){10}{\line(1,0){.04362}}
\multiput(24.936,10.97)(.049838,-.023519){9}{\line(1,0){.049838}}
\multiput(25.385,10.758)(.05744,-.02333){8}{\line(1,0){.05744}}
\multiput(25.844,10.571)(.067015,-.023006){7}{\line(1,0){.067015}}
\multiput(26.313,10.41)(.079544,-.022493){6}{\line(1,0){.079544}}
\put(26.791,10.275){\line(1,0){.484}}
\put(27.275,10.167){\line(1,0){.4892}}
\put(27.764,10.085){\line(1,0){.493}}
\put(28.257,10.031){\line(1,0){.4952}}
\put(28.752,10.003){\line(1,0){.496}}
\put(29.248,10.003){\line(1,0){.4952}}
\put(29.743,10.031){\line(1,0){.493}}
\put(30.236,10.085){\line(1,0){.4892}}
\put(30.725,10.167){\line(1,0){.484}}
\multiput(31.209,10.275)(.079544,.022493){6}{\line(1,0){.079544}}
\multiput(31.687,10.41)(.067015,.023006){7}{\line(1,0){.067015}}
\multiput(32.156,10.571)(.05744,.02333){8}{\line(1,0){.05744}}
\multiput(32.615,10.758)(.049838,.023519){9}{\line(1,0){.049838}}
\multiput(33.064,10.97)(.04362,.023606){10}{\line(1,0){.04362}}
\multiput(33.5,11.206)(.038412,.023612){11}{\line(1,0){.038412}}
\multiput(33.923,11.466)(.037053,.025692){11}{\line(1,0){.037053}}
\multiput(34.33,11.748)(.032616,.025386){12}{\line(1,0){.032616}}
\multiput(34.722,12.053)(.0287708,.0250565){13}{\line(1,0){.0287708}}
\multiput(35.096,12.378)(.0253935,.0247031){14}{\line(1,0){.0253935}}
\multiput(35.451,12.724)(.0258398,.0280695){13}{\line(0,1){.0280695}}
\multiput(35.787,13.089)(.0242542,.0294503){13}{\line(0,1){.0294503}}
\multiput(36.102,13.472)(.024478,.033303){12}{\line(0,1){.033303}}
\multiput(36.396,13.872)(.024661,.037747){11}{\line(0,1){.037747}}
\multiput(36.667,14.287)(.024799,.042953){10}{\line(0,1){.042953}}
\multiput(36.915,14.716)(.024883,.049171){9}{\line(0,1){.049171}}
\multiput(37.139,15.159)(.024904,.056775){8}{\line(0,1){.056775}}
\multiput(37.338,15.613)(.024844,.066356){7}{\line(0,1){.066356}}
\multiput(37.512,16.078)(.024676,.078894){6}{\line(0,1){.078894}}
\multiput(37.66,16.551)(.024351,.09616){5}{\line(0,1){.09616}}
\put(37.782,17.032){\line(0,1){.4868}}
\put(37.877,17.519){\line(0,1){.4913}}
\put(37.945,18.01){\line(0,1){.4943}}
\put(37.986,18.504){\line(0,1){.4958}}
\put(20,19){\circle*{1}}
\put(5,19){\circle*{1}}
\put(38,19){\circle*{1}}
\put(41,2){\circle*{1}}
\put(45,2){\circle*{1}}
\put(49,2){\circle*{1}}
\put(53,19){\circle*{1}}
\put(53,19){\line(-1,0){15}}
\put(20,19){\line(-1,0){15}}
\put(29,19){\makebox(0,0)[cc]{$B'$}}
\put(12,21){\makebox(0,0)[cc]{$B$}}
\put(17,16){\makebox(0,0)[cc]{$v$}}
\put(41,22){\makebox(0,0)[cc]{$w$}}
\put(50,22){\makebox(0,0)[cc]{$x$}}
\put(65.5,24){\oval(17,6)[l]}
\put(65.5,14){\oval(17,6)[l]}
\put(65,24){\makebox(0,0)[cc]{$K_1$}}
\put(65,14){\makebox(0,0)[cc]{$K_r$}}
\put(7,14){\makebox(0,0)[cc]{$G_0$}}
\put(73,19){\makebox(0,0)[cc]{$G_1$}}
\put(45,2){\oval(14,4)[]}
\put(63,19){\oval(30,20)[]}
\put(22,19){\oval(42,20)[]}
\put(63,19.5){\makebox(0,0)[cc]{$\vdots$}}
\put(53,19){\line(3,4){3}}
\put(53,19){\line(2,1){4}}
\put(53,19){\line(2,-1){4}}
\put(53,19){\line(3,-4){3}}
\multiput(49,2)(-.0259433962,.0400943396){424}{\line(0,1){.0400943396}}
\multiput(38,19)(.0259259259,-.062962963){270}{\line(0,-1){.062962963}}
\multiput(38,19)(.025862069,-.146551724){116}{\line(0,-1){.146551724}}
\put(55,2){\makebox(0,0)[cc]{$S$}}
\end{picture}
\end{center}\vspace{-5mm}
\caption{Local structure of $G$.}\label{figg0g1}
\end{figure}
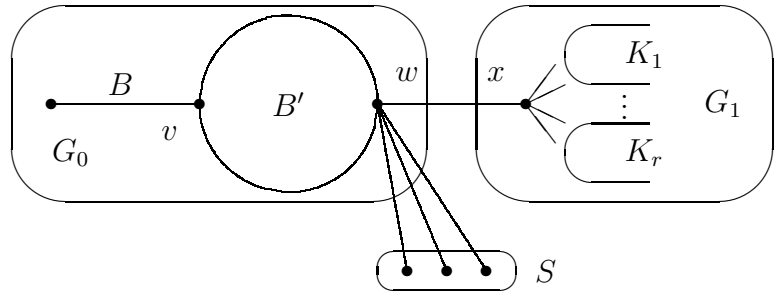
First, we assume that $n_0\equiv 2\,{\rm mod}\,3$.
We now show that $G_0$ has a dissociation set $D_0$ of order 
$\frac{2n_0-1}{3}=n_0-\frac{1}{3}\Big(m(G_0)+1\Big)$
that does not contain $w$.
If $G_0$ is good, then Lemma \ref{lemma1} implies the existence of $D_0$.
If $G_0$ is not good, then, by the parity of $n_0$, Lemma \ref{lemma1} implies
${\rm diss}(G_0)\geq \frac{2n_0-1}{3}+1$,
which also implies the existence of $D_0$.
Note that, in the latter case, the set $D_0$ is not a maximum dissociation set of $G_0$.
Using $D_0$ and Theorem \ref{theorem1}, we obtain 
\begin{eqnarray*}
{\rm diss}(G) 
&\geq & s+\left(n_0-\frac{1}{3}\Big(m(G_0)+1\Big)\right)+{\rm diss}(G_1)\\
&\geq & s+\left(n_0-\frac{1}{3}\Big(m(G_0)+1\Big)\right)
+\left(n_1-\frac{1}{3}\Big(m(G_1)+1+c_1(G_1)\Big)\right)\\
&\geq & n-\frac{1}{3}\Big((m-s)+1+c_1\Big)\\
&\geq & n-\frac{1}{3}\Big(m+1+c_1\Big).
\end{eqnarray*}
Since equality holds throughout this inequality chain,
we obtain that the graph $G_1$ is extremal, 
that $s=0$, and that $c_1(G_1)=c_1$.
By induction, we obtain $G_1\in {\cal C}$.
If $G_0$ is not good, then the union of a maximum dissociation set of $G_0$
and a maximum dissociation set of $G_1$ that does not contain $x$,
cf.~Lemma \ref{lemma2}, yields the contradiction that $G$ is not extremal.
Hence, the spiked cycle $G_0$ is good, and, since $c_1(G_1)=c_1$,
it is very good.
Since $G$ is constructed by applying operation $(O_4)$ to $G_1$,
and we obtain $G\in {\cal C}$.
Hence, we may assume that $n_0\not\equiv 2\,{\rm mod}\,3$.
If $n_0\equiv 1\,{\rm mod}\,3$ and $s\geq 1$,
then exactly the same argument can be repeated 
with $G_0$ and $S$ replaced by $G_0'$ and $S'$,
where the spiked cycle $G_0'$ with at least two spikes 
arises from $G_0$ by attaching one vertex from $S$ to $w$, and 
$S'$ is the set of the remaining $s-1$ vertices from $S$.
Note that $G_0'$ has order $n_0+1\equiv 2\,{\rm mod}\,3$ in that case.
Hence, if $n_0\equiv 1\,{\rm mod}\,3$, then we may assume $s=0$.

Next, we assume that $n_0\equiv 0\,{\rm mod}\,3$.
The tree $T=G_0-w$ has a dissociation set of order
$\left\lceil \frac{2(n_0-1)}{3}\right\rceil=\frac{2n_0}{3}=n_0-\frac{m(G_0)}{3}$.
By Theorem \ref{theorem1}, we obtain the contradiction
\begin{eqnarray*}
{\rm diss}(G) 
&\geq & s+{\rm diss}(T)+{\rm diss}(G_1)\\
&\geq & s+\left(n_0-\frac{m(G_0)}{3}\right)
+\left(n_1-\frac{1}{3}\Big(m(G_1)+1+c_1(G_1)\Big)\right)\\
&>& n-\frac{1}{3}\Big(m+1+c_1\Big).
\end{eqnarray*}
Hence, we may assume that $n_0\equiv 1\,{\rm mod}\,3$,
which implies $s=0$.

Let $G_1'=G_1-x$ have $r$ components $K_1,\ldots,K_r$;
see Figure \ref{figg0g1} for an illustration.
Clearly, we have $r\leq d_G(x)-1$.
The graph $G_0'=G-V(G_1')$ is a spiked cycle with at least two spikes 
that arises from $G_0$ by attaching $x$ to $w$.
Since the order of $G_0'$ is $n_0+1\equiv 2\,{\rm mod}\,3$,
we obtain, similarly as in the case ``$n_0\equiv 2\,{\rm mod}\,3$'', 
that $G_0'$ has a dissociation set $D_0'$ 
of order $n(G_0')-\frac{1}{3}\Big(m(G_0')+1\Big)$ 
that does not contain $x$.
Using $D_0'$ and Theorem \ref{theorem1}, we obtain
\begin{eqnarray*}
{\rm diss}(G) 
&\geq & 
n(G_0')-\frac{1}{3}\Big(m(G_0')+1\Big)
+{\rm diss}(K_1)+\cdots+{\rm diss}(K_r)\\
&\geq & 
n(G_0')-\frac{1}{3}\Big(m(G_0')+1\Big)
+\sum_{i=1}^r \left(n(K_i)-\frac{1}{3}\Big(m(K_i)+1+c_1(K_i)\Big)\right)\\
&=& 
n-\frac{1}{3}\Big((m-d_G(x)+1)+r+1+(c_1(K_1)+\cdots+c(K_r))\Big)\\
&\geq & 
n-\frac{1}{3}\Big(m+1+(c_1(K_1)+\cdots+c(K_r))\Big)\\
&\geq & n-\frac{1}{3}\Big(m+1+c_1\Big).
\end{eqnarray*}
Since equality holds throughout this inequality chain,
we obtain that 
$r=d_G(x)-1$, which implies that every component of $G_1'$ is connected to $x$ by a bridge,
that each $K_i$ is extremal, which, by induction, implies that $K_i\in {\cal C}$, 
and that $c_1(K_1)+\cdots+c(K_r)=c_1$.
If $G_0'$ is not good, then the union of a maximum dissociation set of $G_0'$
and maximum dissociation sets of the $K_i$ that do not contain the neighbors of $x$,
cf.~Lemma \ref{lemma2}, yields the contradiction that $G$ is not extremal.
Hence, the spiked cycle $G'_0$ is good, and, 
since $c_1(K_1)+\cdots+c(K_r)=c_1$, 
it is very good.
If $r=1$, then $G$ is constructed by applying operation $(O_4)$ to $K_1$
and we obtain $G\in {\cal C}$.
Hence, we may assume that $r\geq 2$.
By symmetry, considering alternative choices for the path $P$ 
as well as the previous arguments,
we may assume that $K_r$ is 
either a $P_3$
or a cycle
or a very good spiked cycle,
and that $G-V(K_r)$ is in ${\cal C}$.
It follows that $G$ is constructed by applying 
one of the four operations $(O_1)$ to $(O_4)$ to $G-V(K_r)$.
Hence, we obtain $G\in {\cal C}$,
which completes the proof.
\end{proof}
Within our results, the value $c_1$ can be replaced 
by the maximum number of pairwise vertex-disjoint cycles of length $1$ modulo $3$.
It remains to elucidate the structure of all extremal graphs for Theorem \ref{theorem1}.

\end{document}